\documentclass[12pt,a4paper,openright]{article}
\usepackage[a4paper,textwidth=180mm]{geometry}
\usepackage[OT1]{fontenc}
\usepackage[utf8]{inputenc}
\usepackage{xcolor,lastpage,amsmath,amsthm,amssymb,mathrsfs,enumerate,graphicx,makeidx,hyperref,indentfirst,cite,color,microtype,mathtools,color,titlefoot}
\theoremstyle{definition}
\newtheorem{defn}{Definition}[section]

\newtheorem{lema}[defn]{Lemma}
\newtheorem{teo}[defn]{Theorem}

\newtheorem{obs}[defn]{Remark}

\DeclareMathOperator{\End}{End}
\DeclareMathOperator{\inte}{int}

\DeclareMathOperator{\dive}{div}
\DeclareMathOperator{\tr}{tr}

\DeclareMathOperator{\hess}{Hess}

\DeclareMathOperator{\sn}{sn}
\DeclareMathOperator{\cn}{cn}
\title{Stable free boundary surfaces with constant extrinsic curvature in $3$-dimensional space forms}
\author{Leonardo Damasceno\thanks{Instituto de Matemática, Universidade Federal do Rio de Janeiro, CP 68530, CEP 21941-909, Rio de Janeiro, Brazil. E-mail: \texttt{damasceno@im.ufrj.br}}\and Maria Fernanda Elbert\thanks{Instituto de Matemática, Universidade Federal do Rio de Janeiro, CP 68530, CEP 21941-909, Rio de Janeiro, Brazil. E-mail: \texttt{fernanda@im.ufrj.br}}}
\everymath{\displaystyle}
\allowdisplaybreaks
\begin{document}
\maketitle

\begin{abstract}
{In this paper we use the notion of stability for free boundary surfaces with constant higher order mean curvature to obtain rigidity results for $H_2$-surfaces with free boundary of a geodesic ball of a simply connected $3$-dimensional space form or a slab of $\mathbb{R}^3$.}
\end{abstract}

\unmarkedfntext{Keywords: Isometric immersions, free boundary, capillary, higher order mean curvature}
\unmarkedfntext{2000 Mathematical Subject Classification: 53C42, 53A10.}
\unmarkedfntext{The authors were partially supported by CAPES.}

\section{Introduction}\label{um}

{The extrinsic curvature $H_2$ of a surface is the product of its principal curvatures. As a consequence of the Gauss equation, when immersed into a $3$-manifold with constant curvature equal to $c \in \mathbb{R}$, the intrinsic curvature $K$ and the extrinsic curvature are related via $H_2=K-c$. In particular, when the ambient space is the $3$-dimensional Euclidean space $\mathbb{R}^3$ both notions coincide.

The notion of stability of surfaces with constant mean curvature has been studied by mathematicians throughout the last four decades. It is known that minimal hypersurfaces can be seen as critical points of the volume functional, whereas the hypersurfaces with constant mean curvature (CMC) can also be described on a variational setting. They are critical points of the area functional with respect to variations which preserve volume. Stability, then, means that they are a minimum of area for such variations.

Given a region $\Omega$ of a Riemannian manifold $M$, a hypersurface $\Sigma \subseteq \Omega$ whose boundary $\partial\Sigma$ is contained into $\partial\Omega$ is said to have free boundary if its boundary intersects orthogonally $\partial\Omega$. In a more general situation, when the contact angle is constant along the intersection, such submanifold is said to be capillary. Minimal or CMC capillary hypersurfaces supported on $\partial\Omega$ can also be seen as critical points of the volume functional but, in this case, it is only considered variations which maintains the boundary on the hypersurface which supports it. A number of results has been proved for the case of the ambient space has dimension equal to $3$ \cite{souam1997stability, ros1995stability, nunes2017stable, ros1997stability}.

When considering higher order mean curvatures, the notion of stability does not come from a variational setting. Despite that, for closed hypersurfaces (fixed bounded variations) the second author and Barbara Nelli proposed a notion of stability for such hypersurfaces by {using the} linearization of the corresponding PDE (see \cite{elbert2019note}) and, in \cite{damasceno2022stable}, the first and the second authors proposed a notion of stability for the free boundary and capillary cases).

The main goal of this paper is to prove the following results:
\begin{teo}\label{003.1}
Let $\Sigma^2$ be a closed disk and $\varphi : \Sigma \rightarrow \Omega \subseteq \mathbb{M}^3(c)$ be a $H_2$-surface with free boundary in $\partial\Omega$ and $H_2>0$. Then $\varphi(\Sigma)$ is totally umbilical.
\end{teo}

\begin{teo}\label{006.1}
Let $\varphi : \Sigma^2 \rightarrow B_R \subseteq \mathbb{M}^3(c)$ be a stable $H_2$-surface with free boundary in a geodesic ball $B_R$ with radius $R>0$. If $c>0$ assume the surface is contained into a hemisphere and if $c<0$ assume that $\frac{A(\Sigma)}{\ell(\partial\Sigma)}>-\frac{\cn_c(R)}{c\sn_c(R)}$, where
\begin{equation}\label{0014}
\sn_c(\rho)=\begin{dcases*}
\frac{\sin\left(\rho\sqrt{c}\right)}{\sqrt{c}},& if $c>0$ \\
\rho,& if $c=0$ \\
\frac{\sinh\left(\rho\sqrt{-c}\right)}{\sqrt{-c}},& if $c<0$
\end{dcases*}
\end{equation}
and $\cn_c(\rho)=\sn_c^\prime(\rho)$. Then $\varphi(\Sigma)$ is totally umbilical.
\end{teo}

\begin{teo}\label{008.1}
Let $\varphi : \Sigma \rightarrow \mathbb{R}^3$ be a compact stable $H_2$-surface with a free boundary in a slab bounded by two parallel planes $\Pi_1$ and $\Pi_2$ such that its genus is equal to $0$ and $H_2>0$. Then $\varphi(\Sigma)$ is a surface of revolution around an axis orthogonal to $\Pi_1$.
\end{teo}

The first is a generalization of \cite[Theorem 1]{nitsche1985stationary} when $c=0$ and \cite[Theorem 4.1]{souam1997stability} when $c \neq 0$. The second theorem is an extension of \cite[Theorem 5.1]{souam1997stability} whereas the third is an extension of \cite[Theorem 3.1]{ainouz2016stable}. The paper is organized as the following: the Section \ref{dois} is dedicated to fix the notation and the concepts used throughout the rest of the paper, and the sections \ref{tres}, \ref{quatro} and \ref{cinco} are dedicated to the proof of Theorems \ref{003.1}, \ref{006.1} and \ref{008.1}, respectively.}

\section{Preliminaries}\label{dois}

{Let $\left(M^3,g\right)$ be an oriented Riemannian manifold and $\varphi : \Sigma^2 \rightarrow M$ be an oriented surface with unit normal vector field $\eta$ in the normal bundle $\Gamma(N\Sigma)$. Its second fundamental form $\emph{II}$, scalar second fundamental form $\emph{II}_\eta$ and Weingarten operator $A=\left(\emph{II}_\eta\right)^\flat$ are defined, respectively, as
\begin{eqnarray*}
\emph{II}\left(X,Y\right) &=& \left(\overline\nabla_XY\right)^\perp=\left<\overline\nabla_XY,\eta\right>\eta=\emph{II}_\eta\left(X,Y\right)\eta \\
\left<A(X),Y\right> &=& \emph{II}_\eta\left(X,Y\right)=\left<-\overline\nabla_X\eta,Y\right>,
\end{eqnarray*}
where $X,Y \in \Gamma(T\Sigma)$ and $\overline\nabla $ is the Levi-Civita connection of $M$. Let $\kappa_1(p) \geq \kappa_2(p)$ be the principal curvatures of the surface $\varphi$ at $p$. The $1$-mean curvature $H_1$ is the arithmetic mean of $\kappa_1$ and $\kappa_2$ and the $2$-mean curvature is given by its product $H_2=\kappa_1\kappa_2$. The surface is said to have constant mean curvature of order $r \in \{1,2\}$ if $H_r$ is constant over $\Sigma$; when this happens, $\Sigma$ is called an {$H_r$-surface}.

The first Newton transformation $P_1$ is defined by $P_1=2H_1I-A$. Since $A_p$ is self-adjoint for all $p \in \Sigma$, the Newton transformations are self-adjoint as well and their eigenvectors are the same as those of $A$. If $\{e_1,e_2\}$ denotes the eigenvectors of $A$ then the eigenvalue associated to $e_i$ is equal to $S_1(A_i)=2H_1-\kappa_i$. Moreover, we have the following identities:
\begin{eqnarray}
\tr P_1 &=& 2H_1 \label{0001} \\
\tr P_1A &=& 2H_2 \label{0003} \\
\tr P_1A^2 &=& 2H_1H_2. \label{0008}
\end{eqnarray}

\

In a general Riemannian manifold $(M,g)$ with Levi-Civita connection $\overline\nabla$, if $\phi$ is a pointwise symmetric $(2,0)$-tensor in $M$, the Cheng-Yau operator of $f \in C^\infty(M)$ is defined by
\begin{equation*}
\Box f=\tr\left(\phi\left(\hess f\right)^\flat\right)=\dive\left(\phi\overline\nabla f\right)-\left<\dive\phi,\overline\nabla f\right>,
\end{equation*}
where $\hess f$ is the Hessian of $f$ in $M$, $\left(\hess f\right)^\flat$ is the metric $(1,1)$-tensor field on $M$ equivalent to $\hess f$ and $\dive\phi:=\tr\left(\overline\nabla\phi\right)$. The operator $\phi$ is said to be divergence free if $\dive\phi=0$.

When considering an oriented surface $\varphi : \Sigma^2 \rightarrow M^3$ with shape operator $A \in \Gamma\left(\End\left(T\Sigma\right)\right)$, the $L_1$-operator of $\Sigma$ is defined as the Cheng-Yau operator for the Newton transformation $P_1$, i.e.,
\begin{equation*}
L_1f=\tr\left(P_1\left(\hess f\right)^\flat\right), \quad f \in C^\infty(\Sigma).
\end{equation*}
Here, we say that $-L_1$ is a second-order elliptic differential operator when $P_1$ is positive definite on each point of $\Sigma$. If $H_2>0$, then, after a choice of orientation on $\varphi$, $P_1$ is positive definite. \cite[Lemma 3.10]{elbert2002constant}. In \cite[Theorem 4.1]{rosenberg1993hypersurfaces}, H. Rosenberg proved that $P_1$ is divergence free when $M$ has constant sectional curvature (see also \cite[Corollary 3.7]{elbert2002constant} for the case where $r=1$ and $M$ is Einstein). 
\bigskip

Let $\Omega \subseteq M$ be a closed domain with smooth boundary $\partial\Omega$ and assume that $\varphi : \Sigma \rightarrow M$ is an oriented surface such that $\varphi(\Sigma) \subseteq \Omega$ and $\varphi(\partial\Sigma) \subseteq \partial\Omega$. Let $\nu \in \Gamma\left(T\Sigma\vert_{\partial\Sigma}\right)$ be the unit outward conormal vector field on $\partial\Sigma$ and let $\overline\nu \in \Gamma\left(T\partial\Omega\vert_{\partial\Sigma}\right)$ and $\overline\eta \in \Gamma\left(TM\vert_{\partial\Omega}\right)$ be the unit normal vector fields associated to the immersions $\varphi\vert_{\partial\Sigma} : \partial\Sigma \rightarrow \partial\Omega$ and $\iota_{\partial\Omega} : \partial\Omega \hookrightarrow M$, respectively, such that $\left\{\nu,\eta\right\}$ has the same orientation as $\left\{\overline\nu,\overline\eta\right\}$ on each point of $\varphi(\partial\Sigma)$. If $\theta$ denotes the angle between $\nu$ and $\overline\nu$, then
\begin{equation}\label{0002}\begin{dcases*}
\nu=\cos\theta\,\overline\nu+\sin\theta\,\overline\eta \\
\eta=-\sin\theta\,\overline\nu+\cos\theta\,\overline\eta
\end{dcases*}, \quad \text{or conversely,} \quad \begin{dcases*}
\overline\nu=\cos\theta\,\nu-\sin\theta\,\eta \\
\overline\eta=\sin\theta\,\nu+\cos\theta\,\eta
\end{dcases*}.
\end{equation}
A $H_r$-surface $\varphi : \Sigma \rightarrow \Omega \subseteq M$ is said to be {capillary} if the contact angle $\theta$ between $\partial\Sigma$ and $\partial\Omega$ is constant. When $\theta=\frac{\pi}{2}$, $\varphi$ is called a {free boundary surface}. When $\varphi : \Sigma \rightarrow \Omega \subseteq M$ is a surface with free boundary in $\Omega$, \eqref{0002} implies that $\nu=\overline\eta$ and $\eta=-\overline\nu$.

The following result will be used throughout this article.

\begin{lema}[{\cite[Lemma 2.2]{ainouz2016stable}}]\label{002.4}
Suppose $\iota_{\partial\Omega}$ is a totally umbilical immersion into $M$ and that $\varphi$ is a capillary immersion into $M$. Then the unit outwards normal vector field $\nu \in \Gamma\left(T\Sigma\vert_{\partial\Sigma}\right)$ is a principal direction of $\varphi$.
\end{lema}

A variation of $\varphi$ is a smooth function $\Phi : \Sigma^2 \times (-\varepsilon,\varepsilon) \rightarrow M^3$ such that, for each $t \in (-\varepsilon,\varepsilon)$, $\varphi_t=\Phi\vert_{\Sigma \times \{t\}}$ is an isometric immersion and $\varphi_0=\varphi$. The pair $(\Sigma,\varphi_t^*g)$ will be denoted by $\Sigma_t$. The variational field of $\Phi$ in $\varphi_t$ is defined by $$\xi_t(p)=\left.\Phi_*\frac{\partial}{\partial t}\right\vert_{(p,t)} \in \Gamma\left(TM\vert_{\varphi_t(\Sigma)}\right).$$ If $\eta_t \in \Gamma(N\Sigma)$ is the unit normal vector field of $\varphi_t$, the support function of $\Phi$ at $t$ is defined by $$f_t=\left<\xi_t,\eta_t\right> \in C^\infty(\Sigma).$$ Since $\varphi_t : \Sigma \rightarrow M$ is an oriented surface, one can define its second fundamental form $\emph{II}_t$, its scalar second fundamental form $\left(\emph{II}_t\right)_{\eta_t}$ and its Weingarten operator $A_t$. Also, we set $\overline{R}_{\eta_t}\left(X\right):=\overline{R}\left(\eta_t,X\right)\eta_t$, where $\overline{R}$ the Riemann curvature tensor of $M$ defined by $$\overline{R}(X,Y)Z=\overline\nabla_Y\overline\nabla_XZ-\overline\nabla_X\overline\nabla_YZ+\overline\nabla_{[X,Y]}Z, \quad X,Y,Z \in \Gamma(TM).$$ If $H_2(t)$ denotes the $2$-mean curvature associated to immersion $\varphi_t$, its variation is given by
\begin{equation}\label{0004}
H_2^\prime(t)=\left(L_1\right)_tf_t+2H_1(t)H_2(t)f_t+\tr_{\Sigma_t}\left(\left(P_1\overline{R}_\eta\right)_t\right)f_t+\xi_t^\top\left(H_2(t)\right),
\end{equation}
where $\left(L_1\right)_t$ is the $L_1$-operator of immersion $\varphi_t$ and $\left(P_1\overline{R}_\eta\right)_t:=\left(P_1\right)_t \circ \overline{R}_{\eta_t}$. A proof of \eqref{0004} can be found in \cite[Proposition 3.2]{elbert2002constant}.

The enclosed volume between $\Sigma$ and $\Sigma_t$ is defined as $\mathcal{V}(t)=\int_{\Sigma \times [0,t]} \Phi^*d\mu_M$, with $d\mu_M$ being the volume form of $(M,g)$. A variation $\Phi$ is volume-preserving if $\mathcal{V}(t)=\mathcal{V}(0)$ for all $t \in (-\varepsilon,\varepsilon)$. It is known that $$\mathcal{V}^\prime(0)=\int_\Sigma f\,d\mu_\Sigma,$$ where $u=\left<\xi,\eta\right> \in C^\infty(\Sigma)$ and $d\mu_\Sigma$ is the volume form of $\left(\Sigma,\varphi^*g\right)$. Thus, a variation $\Phi$ is volume-preserving if and only if $\int_\Sigma f\,d\mu_\Sigma=0$.



A $H_2$-surface $\varphi : \Sigma \rightarrow M$ is {positive definite} if $P_1$ is positive definite on each point $p \in \Sigma$. A variation $\Phi$ of a surface $\varphi : \Sigma \rightarrow \Omega \subseteq M$ is called admissible if $\varphi_t(\inte\Sigma) \subseteq \inte\Omega$ and $\varphi_t(\partial\Sigma) \subseteq \partial\Omega$ for any $t \in (-\varepsilon,\varepsilon)$, where $\varphi_t=\Phi\vert_{\Sigma \times \{t\}}$. If $\Phi$ is an admissible variation of $\varphi$, then $\xi\vert_{\partial\Sigma} \in \Gamma\left(T\partial\Omega\vert_{\partial\Sigma}\right)$. If $\Sigma$ is a capillary $H_2$-surface supported in $\partial\Omega$ with contact angle $\theta \in (0,\pi)$ and $\Phi$ is a volume-preserving admissible variation of $\varphi$, consider the functional, defined in \cite{damasceno2022stable},
\begin{equation}\label{0005}
\mathcal{F}_{1,\theta}[\Sigma_t]=-\int_\Sigma H_2(t)\left<\xi_t,\eta_t\right>\,d\mu_{\Sigma_t}+\int_{\partial\Sigma} \left<\xi_t,(P_1\nu-\vert{P_1\nu}\vert\cos\theta\,\overline\nu)_t\right>\,d\mu_{\partial\Sigma_t},
\end{equation}
where $d\mu_{\Sigma_t}$ and $d\mu_{\partial\Sigma_t}$ denote the volume forms of $\Sigma_t$ and $\partial\Sigma_t=\left(\partial\Sigma,\left(\varphi_t\vert_{\partial\Sigma}\right)^*g\right)$, respectively. If $\partial\Omega$ is totally umbilical and $\Phi$ is an admissible volume-preserving variation of $\varphi$ then 
\begin{multline}\label{0006}
\left.\frac{\partial}{\partial t}\mathcal{F}_{1,\theta}\left[\Sigma_t\right]\right\vert_{t=0}=-\int_\Sigma f\left(L_1f+\tr\left(P_1\left(A^2+\overline{R}_\eta\right)\right)f\right)\,d\mu_\Sigma+\\+\int_{\partial\Sigma} \left\vert{P_1\nu}\right\vert\,f\left(\frac{\partial f}{\partial\nu}+\left(\csc\theta\left(\emph{II}_{\partial\Omega}\right)_{\overline\eta}(\overline\nu,\overline\nu)-\cot\theta\left(\emph{II}_\Sigma\right)_\eta(\nu,\nu)\right)f\right)\,d\mu_{\partial\Sigma},
\end{multline}
where $f=\left<\xi,\eta\right> \in C^\infty(\Sigma)$ is the support function of $\Phi$ at $t=0$ and $\emph{II}_\Sigma$ and $\emph{II}_{\partial\Omega}$ are the second fundamental forms of $\varphi : \Sigma \rightarrow \Omega$ and $\iota_{\partial\Omega} : \partial\Omega \hookrightarrow \Omega$, respectively. For a proof see \cite[Appendix A]{damasceno2022stable}. A positive definite capillary $H_2$-surface $\varphi : \Sigma \rightarrow \Omega \subseteq M$ supported in $\partial\Omega$ with contact angle $\theta \in (0,\pi)$ is {$r$-stable} if $\left.\frac{\partial}{\partial t}\mathcal{F}_{1,\theta}\left[\Sigma_t\right]\right\vert_{t=0} \geq 0$ for any volume-preserving admissible variation $\Phi$ of $\varphi$. If the inequality holds for all admissible variations of $\varphi$, $\Sigma$ is said to be {strongly $r$-stable}. The expression \eqref{0006} is associated to the eigenvalue problem below:
\begin{equation}\label{0007}\begin{dcases*}
T_1f=-L_1f-q_rf=\lambda f, & \quad $\text{in}~\Sigma$ \\
\frac{\partial f}{\partial\nu}+\alpha_\theta f=0, & \quad $\text{on}~\partial\Sigma$
\end{dcases*},\end{equation}
where $q_r=\tr\left(P_1\left(A^2+\overline{R}_\eta\right)\right) \in C^\infty(\Sigma)$ and $\alpha_\theta=\csc\theta\left(\emph{II}_{\partial\Omega}\right)_{\overline\eta}(\overline\nu,\overline\nu)-\cot\theta\left(\emph{II}_\Sigma\right)_\eta(\nu,\nu) \in C^\infty(\partial\Sigma)$. For the properties involving its principal eigenvalue see \cite[Proposition 3.4]{damasceno2022stable}. The notion of $1$-stability can also be considered when $P_1$ is negative definite, see \cite[Remark 3.5]{damasceno2022stable}.

Let $\mathbb{M}^3(c)$ be the simply connected space form of constant sectional curvature $c$, i.e., $\mathbb{M}^3(c)$ is equal to $\mathbb{R}^3$ if $c=0$, $\mathbb{S}^3(c)$ if $c>0$ and $\mathbb{H}^3(c)$ if $c=0$. In this paper we consider the following models for $\mathbb{M}^3(c)$:
\begin{eqnarray*}
\mathbb{R}^3 &=& \left\{x=(x_1,x_2,x_3,x_4) \in \mathbb{R}^4\,\left|\right.\,x_4=0\right\} \\
\mathbb{S}^3(c) &=& \left\{x=(x_1,x_2,x_3,x_4) \in \mathbb{R}^4\,\left|\right.\,x_1^2+x_2^2+x_3^2+x_4^2=\frac{1}{c^2}\right\} \\
\mathbb{H}^3(c) &=& \left\{x=(x_1,x_2,x_3,x_4) \in \mathbb{R}_1^4\,\left|\right.\,x_1^2+x_2^2+x_3^2-x_4^2=-\frac{1}{c^2}, x_4>0\right\}
\end{eqnarray*}
endowed with the pullback of the Euclidean metric for $c \geq 0$ or the Minkowski metric for $c<0$. When $M=\mathbb{M}^3(c)$ we have that $\overline{R}_\eta(X)=cX$ for all $X \in \Gamma(T\mathbb{M}^3(c))$ and $\dive P_1=0$ (see \cite[Theorem 4.1]{rosenberg1993hypersurfaces}). Thus, \eqref{0006} can be rewritten as
\begin{equation}\label{0009}
\left.\frac{\partial}{\partial t}\mathcal{F}_{1,\theta}\left[\Sigma_t\right]\right\vert_{t=0}=\int_\Sigma \left<P_1\nabla f,\nabla f\right>-2H_1\left(H_2+c\right)f^2\,d\mu_\Sigma+\int_{\partial\Sigma}\left\vert{P_1\nu}\right\vert\alpha_\theta f^2\,d\mu_{\partial\Sigma}.
\end{equation}
The equation \eqref{0009} can be viewed as a quadratic form associated to a bilinear symmetric form on the Hilbert space $H^1(\Sigma)$, the closure of $C^\infty(\Sigma)$ with respect to the Sobolev norm $$\left\Vert\cdot\right\Vert_{H^1(\Sigma)}^2=\left\Vert\cdot\right\Vert_{L^2(\Sigma)}^2+\left\Vert\nabla\cdot\right\Vert_{L^2(\Sigma)}^2.$$ The $1$-index form of $\varphi : \Sigma \rightarrow \Omega\subseteq\mathbb{M}^3(c)$ is
\begin{equation}\label{0010}
\mathcal{I}_{1,\theta}(f_1,f_2)=\int_\Sigma \left<P_1\nabla f_1,\nabla f_2\right>-2H_1\left(H_2+c\right)f_1f_2\,d\mu_\Sigma+\int_{\partial\Sigma}\left\vert{P_1\nu}\right\vert\alpha_\theta f_1f_2\,d\mu_{\partial\Sigma},
\end{equation}
where $f_1,f_2 \in H^1(\Sigma)$. $\Sigma$ is strongly $1$-stable if and only if $\mathcal{I}_{1,\theta}(f,f) \geq 0$ for all $f \in H^1(\Sigma)$ and $1$-stable if $\mathcal{I}_{1,\theta}(f,f) \geq 0$ for all $f \in \mathcal{F}=\left\{f \in H^1(\Sigma)\,|\,\int_\Sigma f\,d\mu_\Sigma=0\right\}$. It can be proved that a totally umbilical capillary compact surface supported on a connected totally umbilical surface of $\mathbb{M}^3(c)$ is $1$-stable \cite[Proposition 4.2]{damasceno2022stable}.

\

As in the case $r=0$, when considering $\varphi$ a capillary $(r+1)$-minimal surface, i.e. $H_2=0$, we say that $\varphi$ is $1$-stable if $\mathcal{I}_{1,\theta}(f,f) \geq 0$ for all $f \in C_0^\infty(\Sigma)$. This means the hypothesis on the variation being volume-preserving is dropped.

\

If $f \in \mathcal{F}$, the normal vector field $\xi=f\eta$ on $\Sigma$ is a {Jacobi field} if $f \in \ker\mathcal{I}_{1,\theta}\vert_{\mathcal{F} \times \mathcal{F}}$, i.e., $\mathcal{I}_{1,\theta}(f,g)=0$ for every $g \in \mathcal{F}$. The next lemma, whose proof is in \cite[Lemma 4.4]{damasceno2022stable}, gives a characterization of Jacobi fields on $\Sigma$.

\begin{lema}\label{005.3}
Let $\varphi : \Sigma \rightarrow \Omega \subseteq \mathbb{M}^3(c)$ be a positive definite $H_2$-surface with free boundary in $\partial\Omega$ and $f \in \mathcal{F}$. Then

\begin{enumerate}[i)]
\item $\xi=f\eta$ is a Jacobi field on $\Sigma$ if and only if $f \in C^\infty(\Sigma)$ and
\begin{equation}\label{0011}\begin{dcases*}
T_1f=\text{constant} & in $\Sigma$ \\
\frac{\partial f}{\partial\nu}+\alpha_\theta f=0 & on $\partial\Sigma$
\end{dcases*}.\end{equation}

\item If $\varphi$ is $r$-stable and $\mathcal{I}_{1,\theta}(f,f)=0$ then $f$ is a Jacobi field on $\Sigma$.
\end{enumerate}
\end{lema}}

\section{Proof of Theorem \ref{003.1}}\label{tres}
{The Theorem \ref{003.1}} is inspired by its analogous version proved by J. Nitsche in \cite{nitsche1985stationary} when $c=0$ and by R. Souam in \cite[Theorem 4.1]{souam1997stability} when $c \neq 0$, both of them addressing constant mean curvature immersions of disk into a ball in $\mathbb{M}^3(c)$. Here, we consider immersions of the disk into a compact, convex smooth body $\Omega$. In order to prove this result, one needs the following theorem proved by R. Bryant.

\

{\noindent {\bf Theorem A }\cite[Theorem 3]{bryant2011complex}:
Let $\varphi : \Sigma^2 \rightarrow \mathbb{M}^3(c)$ be a smooth immersion that satisfies a Weingarten equation of the form $H_1=f(H_1^2-H_2)$, for some $f \in C^\infty((-\varepsilon,\infty);\mathbb{R})$ and $\varepsilon>0$. Then $\varphi$ is totally umbilical or else, the umbilic locus consists entirely of isolated points of strictly negative index.}

For a definition of index, see \cite[p. 107]{hopf2003differential}. As a direct consequence of the Poincaré-Hopf Theorem for manifolds with boundary \cite[p. 35]{milnor1997topology}, we have a boundary version of Hopf's Theorem:

\

\noindent {\bf Theorem B }(Boundary version of Hopf's Theorem)
Let $\Sigma^2$ be a compact manifold with boundary $\partial\Sigma$ for which the umbilic locus $\mathcal{U}$ is finite. Suppose that one of the principal directions is transversal to $\partial\Sigma$. Then $$\sum_{p \in \mathcal{U}} i(p)=\chi(\Sigma),$$ where $i(p)$ is the index of $p \in \mathcal{U}$.

{\begin{proof}[Proof of Theorem \ref{003.1}]
Since $\varphi : \Sigma^2 \rightarrow \mathbb{M}^3(c)$ is a $H_2$-surface, its mean curvature satisfies a Weingarten equation $H_1=f(H_1^2-H_2)$, where $f(y)=\sqrt{y+H_2}$ and $y \in \left(-H_2,+\infty\right)$. If $\varphi$ is not totally umbilic then, by Theorem A, the umbilical points of $\varphi$ form a finite set $\mathcal{U} \subseteq \Sigma$ and each umbilical point has negative index. Since $\nu$ is a principal direction along $\partial\Sigma$, one can use Theorem B to obtain $$0>\sum_{p \in \mathcal{U}} i(p)=\chi(\Sigma)=1,$$ which is a contradiction. Thus $\Sigma$ is totally umbilical.
\end{proof}}

\begin{obs}\label{003.2}
The same argument holds if $\Sigma$ is a CMC surface since Theorem A still holds in this case (choose $f$ to be a constant function).
\end{obs}

\section{Proof of Theorem \ref{006.1}}\label{quatro}
{In this section we will prove Theorem \ref{006.1}. The geodesic ball $B_R$ is convex for all $R \in (0,R_c)$, where $R_c=+\infty$ if $c \leq 0$ and $R_c=\frac{\pi}{2\sqrt{c}}$ if $c>0$. Its boundary $\partial B_R$ is a totally umbilical sphere whose mean curvature with respect to the inward unit normal is equal to $\frac{\cn_c(R)}{\sn_c(R)}$.

One must state some identities that will be used throughout the proof.}

\begin{lema}\label{006.2}
Let $\varphi : \Sigma^2 \rightarrow \mathbb{M}^3(c)$ be a surface. Then
\begin{eqnarray}
L_1\varphi &=& 2H_2\eta-2cH_1\varphi \label{0015} \\
L_1\eta &=& -\tr\left(P_1A^2\right)\eta+2cH_2\varphi-\nabla H_2, \label{0016}
\end{eqnarray}
Here $L_1\varphi$ and $L_1\eta$ are calculated coordinate-wise.
\end{lema}
For a proof of \eqref{0015} and \eqref{0016} see \cite[Remark 5.1]{rosenberg1993hypersurfaces}.

The next Lemma also has key role for the proof of Theorem \ref{006.1}.
\begin{lema}\label{006.3}
Suppose that $\varphi : \Sigma \rightarrow \Omega \subseteq M$ is a surface such that $H_2>0$ and let $u \in C^\infty(\Sigma) \backslash\{0\}$ be a function such that $T_1u=0$. Then its nodal set $\Gamma=u^{-1}(\{0\})$ is a finite graph whose vertices are the critical points of $u$. In a neighborhood of each critical point $\Gamma$ is a star of at least two branches.
\end{lema}

\begin{proof}
Let $p \in \Sigma$ and take $\varphi : U_p \subseteq \mathbb{R}^2 \rightarrow \Sigma$ a parametrization of $\Sigma$ at $p$ with local coordinates $(u,v)$. Since $L_1$ is a second-order elliptic differential operator only with principal part, it follows from PDE theory in \cite[Chapter 3]{epstein1962partial} that there exists a coordinate change $$\overline{u}=h_1(u,v), \quad \overline{v}=h_2(u,v)$$ of class $C^2$ in a neighborhood of $p_0=\varphi^{-1}(p)$ whose Jacobian does not vanish at $p_0$ that transforms the pullback of $L_1$ in the Laplacian operator. We may suppose (restricting $U_p$ if necessary) that $\left(\overline{u},\overline{v}\right)$ is a diffeomorphism in $U_p$. Thus, in the new coordinates $(\overline{u},\overline{v})$, $L_1$ is the Laplacian and \cite[Theorem 2.5]{cheng1976eigenfunctions} implies that its nodal lines in $U_p$ meet at the critical points. Since $\Sigma$ is compact, we can cover $\Sigma$ with finitely many such open neighborhoods $U_p$, proving the claim.
\end{proof}

\begin{proof}[Proof of Theorem \ref{006.1}]
{The proof is an extension of that in \cite[Theorem 11]{ros1995stability}. From \cite[Lemma 1.1]{ros1997stability}, $\left(\emph{II}_{\partial B}\right)_{\overline\eta}(\overline\nu,\overline\nu)=-\frac{\cn_c(R)}{\sn_c(R)}$ and the curvature of $\partial\Sigma$ in $\Sigma$ is equal to $\frac{\cn_c(R)}{\sn_c(R)}$. Thus, the Gauss-Bonnet Theorem implies that $$2\pi\chi(\Sigma)=\int_\Sigma K\,d\mu_\Sigma+\int_{\partial\Sigma}\kappa_g\,d\mu_{\partial\Sigma}=\int_\Sigma H_2+c\,d\mu_\Sigma+\int_{\partial\Sigma}\kappa_g\,d\mu_{\partial\Sigma} >cA(\Sigma)+\frac{\cn_c(R)}{\sn_c(R)}\ell(\partial\Sigma).$$ In all three cases the inequality above implies that the genus of $\Sigma$ is equal to zero.}

Now consider the case $c=0$ and suppose, without loss of generality, that $B_R$ is the unit ball centered at the origin of $\mathbb{R}^3$. Let $p_0 \in \Sigma$ be a point where the function $p \in \Sigma \mapsto \vert\varphi(p)\vert$ attains its minimum and define the function
\begin{equation}\label{0019}
f(p)=\left<\varphi(p) \wedge \eta(p_0),\eta(p)\right>, \quad p \in \Sigma,
\end{equation}
where $\wedge$ denotes the cross product in $\mathbb{R}^3$. It is clear that $f(p_0)=0$ and for all $\mathbf{v} \in T\Sigma$,
\begin{equation*}
\left<\nabla f,\mathbf{v}\right>= \mathbf{v}\left<\varphi \wedge \eta(p_0),\eta\right>=\left<\mathbf{v} \wedge \eta(p_0),\eta\right>+\left<\varphi \wedge \eta(p_0),\overline\nabla_{\mathbf{v}}\eta\right>=\left<\eta(p_0) \wedge \eta-A\left(\varphi \wedge \eta(p_0)\right)^\top,\mathbf{v}\right>.
\end{equation*}
Thus $\nabla f=\eta(p_0) \wedge \eta-A\left(\varphi \wedge \eta(p_0)\right)^\top$, where $^\top$ denotes the projection onto $T\Sigma$. Since $\left\vert\varphi\right\vert$ attains its minimum at $p_0$, we have $\varphi(p_0) \parallel \eta(p_0)$, hence $\nabla f(p_0)=0$. Also we have
\begin{eqnarray*}
L_1f &=& L_1\left<\varphi \wedge \eta(p_0),\eta\right>=L_1\left<\eta \wedge \varphi,\eta(p_0)\right> \\
&=& L_1\left<\left(\eta^2\varphi^3-\eta^3\varphi^2,\eta^3\varphi^1-\eta^1\varphi^3,\eta^1\varphi^2-\eta^2\varphi^1\right),\eta(p_0)\right> \\
&=& \left<\left(L_1\left(\eta^2\varphi^3-\eta^3\varphi^2\right),L_1\left(\eta^3\varphi^1-\eta^1\varphi^3\right),L_1\left(\eta^1\varphi^2-\eta^2\varphi^1\right)\right),\eta(p_0)\right>,
\end{eqnarray*}
where $\varphi^i=\left<\varphi,\mathbf{e}_i\right>$, $\eta^i=\left<\eta,\mathbf{e}_i\right>$, $i \in \{1,2,3\}$, and the vectors $\{\mathbf{e}_1,\mathbf{e}_2,\mathbf{e}_3\}$ form the canonical basis of $\mathbb{R}^3$. Using \eqref{0015} and \eqref{0016}, we have for $i,j \in \{1,2,3\}$,
\begin{eqnarray*}
L_1\left(\eta^i\varphi^j\right) &=& \varphi^jL_1\eta^i+\eta^iL_1\varphi^j+2\left<P_1\nabla\eta^i,\nabla\varphi^j\right> \\
&=& -2H_1H_2\eta^i\varphi^j+2H_2\eta^i\eta^j-2\left<P_1A\mathbf{e}_i^\top,\mathbf{e}_j^\top\right> \\
L_1\left(\eta^i\varphi^j-\eta^j\varphi^i\right) &=& -2H_1H_2\left(\eta^i\varphi^j-\eta^j\varphi^i\right).
\end{eqnarray*}
Thus, $L_1f+2H_1H_2f=0$ on $\Sigma$. Moreover, since $\varphi=\nu$ on $\partial\Sigma$, the Lemma \ref{002.4} implies that, on $\partial\Sigma$
\begin{eqnarray*}
\frac{\partial f}{\partial\nu} &=& \left<\nabla f,\nu\right>=\left<\eta(p_0) \wedge \eta-A(\varphi \wedge \eta(p_0))^\top,\nu\right> \\
&=& \left<\nu \wedge \eta(p_0),\eta\right>-\left<\nu \wedge \eta(p_0),A\nu\right> \\
&=& f-\left\vert{A\nu}\right\vert\left<\nu \wedge \eta(p_0),\nu\right>=f.
\end{eqnarray*}
Hence the function $f$ satisfies
\begin{equation}\label{0020}
\begin{dcases*}
L_1f+2H_1H_2f=0& in $\Sigma$ \\
\frac{\partial f}{\partial\nu}-f=0& on $\partial\Sigma$
\end{dcases*}.\end{equation}

We claim that $f \equiv 0$ on $\Sigma$. Otherwise, Lemma \ref{006.3} implies that the lines of the nodal set $f^{-1}(\{0\})$ meet at the critical points of $f$. Using the Gauss-Bonnet theorem for each connected component $\Sigma_i$ of $\Sigma \backslash \beta^{-1}(\{0\})$, we have
\begin{equation}\label{0021}
\int_{\Sigma_i} K\,d\mu_\Sigma=2\pi\chi(\Sigma_i)-\int_{\partial\Sigma_i} \kappa_g\,d\mu_{\partial\Sigma}-\sum_j \theta_{ij},
\end{equation}
where $\theta_{ij}$, $j \in \{1,...,j_i\}$ denotes the external angles of $\Sigma_i$. Summing up \eqref{0021} for all $i$, we obtain
\begin{eqnarray}
\int_\Sigma K\,d\mu_\Sigma &=& \sum_i \int_{\Sigma_i} K\,d\mu_\Sigma=\sum_i \left(2\pi\chi(\Sigma_i)-\int_{\partial\Sigma_i} \kappa_g\,d\mu_{\partial\Sigma}-\sum_j \theta_{ij}\right) \nonumber \\
&=& 2\pi\sum_i \chi(\Sigma_i)-\int_{\partial\Sigma} \kappa_g\,d\mu_{\partial\Sigma}-\sum_l \theta_l, \label{0022}
\end{eqnarray}
where the last term means the sum of all external angles for every connected component $\Sigma_i$. Since $\partial\Sigma$ is smooth, it follows again from the Gauss-Bonnet theorem that \eqref{0022} implies to
\begin{eqnarray*}
2\pi\left(2-2g-s\right) &=& 2\pi\chi(\Sigma)=\int_\Sigma K\,d\mu_\Sigma+\int_{\partial\Sigma} \kappa_g\,d\mu_{\partial\Sigma} \\
&=& 2\pi\sum_i \chi(\Sigma_i)-\sum_l \theta_l,
\end{eqnarray*}
where $s$ is the number of components of $\partial\Sigma$. Since $f(p_0)=0$ and $\nabla f(p_0)=0$, the Lemma \ref{006.3} implies there are at least two nodal lines of $f$ intersecting at $p_0$ and forming a star at $p_0$; so $\sum_l \theta_l \geq 2\pi$. On the other hand, on each connected component $\Gamma_i$ of $\partial\Sigma$, $i \in \{1,...,s\}$, choosing a positively oriented arclength parametrization $\gamma$ we have $\varphi \wedge \eta=-\gamma^\prime$. So
\begin{equation*}
\int_{\Gamma_i}f\,d\mu_{\partial\Sigma}=\int_{\Gamma_i} \left<\varphi\wedge\eta(p_0),\eta\right>\,d\mu_{\partial\Sigma}=-\int_{\Gamma_i} \left<\varphi\wedge\eta,\eta(p_0)\right>\,d\mu_{\partial\Sigma}=\int_{\Gamma_i} \left<\gamma^\prime,\eta(p_0)\right>\,d\mu_{\partial\Sigma}=0,
\end{equation*}
and it follows that $f$ has at least two zeroes on each component $\Gamma_i$. Each point of $f^{-1}\left(\{0\}\right) \cap \Gamma_i$ contributes with at least $\pi$ for the sum of the $\theta_j$ in the last equation. Putting things together, we have
\begin{equation}\label{0023}
\sum_l \theta_l \geq 2\pi\left(1+s\right)
\end{equation}
and using \eqref{0023} in \eqref{0022} we obtain $$\sum_i \chi(\Sigma_i)=\frac{1}{2\pi}\left(2\pi\left(2-2g-s\right)+\sum_l \theta_l\right) \geq 2-2g-s+1+s=3-2g.$$ Assuming that $\Sigma$ has genus $g=0$, it follows that $\Sigma\backslash f^{-1}\left(\{0\}\right)$ has at least three connected components. If $\Sigma_1$ and $\Sigma_2$ are two connected components of the nodal domain of $f$, define $$\widetilde{f}=\begin{dcases*}f& in $\Sigma_1$ \\ \alpha f& in $\Sigma_2$ \\ 0& in $\Sigma \backslash (\Sigma_1 \cup \Sigma_2)$\end{dcases*},$$ where $\alpha \in \mathbb{R}$ is such that $\widetilde{f} \in \mathcal{F}$. Since $\partial\Sigma_i \cap \partial\Sigma=\partial\Sigma \cap \Sigma_i$ and $\widetilde{f} \equiv 0$ outside $\Sigma_i$, we have
\begin{eqnarray*}
\int_{\Sigma_1} \left<P_1\nabla\widetilde{f},\nabla\widetilde{f}\right>-2H_1H_2\widetilde{f}^2\,d\mu_\Sigma &=& \int_{\Sigma_1} \left<P_1\nabla\widetilde{f},\nabla f\right>-2H_1H_2\widetilde{f}f\,d\mu_\Sigma \\
&=& -\int_{\Sigma} \widetilde{f}\left(L_1f+2H_1H_2f\right)\,d\mu_\Sigma+\int_{\partial\Sigma \cap \Sigma_1}\vert{P_1\nu}\vert\widetilde{f}\frac{\partial f}{\partial\nu}\,d\mu_{\partial\Sigma} \\
&=& \int_{\partial\Sigma \cap \Sigma_1} \left\vert{P_1\nu}\right\vert \widetilde{f}^2\,d\mu_{\partial\Sigma}
\end{eqnarray*}
and, similarly, $$\int_{\Sigma_2}\left<P_1\nabla\widetilde{f},\nabla\widetilde{f}\right>-2H_1H_2\widetilde{f}^2\,d\mu_\Sigma=\int_{\partial\Sigma \cap \Sigma_2}\left\vert{P_1\nu}\right\vert\widetilde{f}^2\,d\mu_{\partial\Sigma}.$$ Thus, $$\mathcal{I}_1(\widetilde{f},\widetilde{f})=\sum_{i=1}^2 \int_{\Sigma_i} \left<P_1\nabla\widetilde{f},\nabla\widetilde{f}\right>-2H_1H_2\widetilde{f}^2\,d\mu_\Sigma-\int_{\partial\Sigma \cap \Sigma_i}\left\vert{P_1\nu}\right\vert\widetilde{f}^2\,d\mu_{\partial\Sigma}=0.$$ Hence, the second item of Lemma \ref{005.3} implies that $\widetilde{f}$ is a Jacobi field on $\Sigma$. But since $\widetilde{f} \equiv 0$ outside of $\Sigma_1 \cap \Sigma_2$, the Aronszajn's unique continuation principle \cite{aronszajn1956unique} implies that $\widetilde{f} \equiv 0$, which is a contradiction.

Finally, since $f \equiv 0$, the Killing field $p \in \mathbb{R}^3 \mapsto p \wedge \eta(p_0) \in \mathbb{R}^3$ is tangent to $\Sigma$. Hence, $\Sigma$ is a rotation surface around the axis $\eta(p_0)$ with fixed point $p_0$ and thus, $\Sigma$ must be homeomorphic to a disk. Using Theorem \ref{003.1}, we conclude that $\Sigma$ is totally umbilical.

The non-Euclidean cases use similar arguments and, since the spherical and the hyperbolic cases are very similar, we will give a sketch of the proof only when $c=-1$. Using the same notation used in \cite[Theorem 5.1]{souam1997stability}, define $f : \Sigma \rightarrow \mathbb{R}$ by $$f(p)=\left<\varphi(p) \wedge \eta(p_0) \wedge \mathbf{e}_4,\eta(p)\right>.$$ The same arguments used in the Euclidean case gives $\nabla f(p_0)=0$ and $$\begin{dcases*}L_1f+2(H_1H_2-1)f=0 & in $\Sigma$ \\ \frac{\partial f}{\partial\nu}-\frac{\cn_c(R)}{\sn_c(R)}f=0& on $\partial\Sigma$\end{dcases*}.$$ It can also be shown that $f \equiv 0$ and to prove this claim, it is considered a positively oriented arclength parametrization $\gamma$ of a connected component $\Gamma_i$ of $\partial\Sigma$ satisfying $\varphi \wedge \eta \wedge \nu=-\gamma^\prime$. The identity $f \equiv 0$ implies that $\Sigma$ is a rotation surface in $\mathbb{R}^4$ around the plane generated by $\mathbf{e}_4$ and $\eta(p_0)$ with fixed point $p_0$, proving that $\Sigma$ is a disk.
\end{proof}

\section{Proof of Theorem \ref{008.1}}\label{cinco}
In this section we will extend \cite[Theorem 3.1]{ainouz2016stable} for $1$-stable $H_2$-surfaces with free boundary in a slab of $\mathbb{R}^3$, $H_2>0$ and genus $0$.

\begin{proof}[Proof of Theorem \ref{008.1}]
The proof of this result is an adaptation to the arguments used in \cite[Theorem 3.1]{ainouz2016stable}. Without loss of generality, one can suppose that $\Pi_1=\{x_3=0\}$ and $\Pi_2=\{x_3=1\}$. Let $\Gamma$ be a connected component of $\partial\Sigma$ such that $\varphi(\Gamma)$ lies on $\Pi_1$ and consider in this plane the circumscribed circle $\mathscr{C}$ about $\varphi(\Gamma)$. We will prove that $\varphi(\Sigma)$ is a surface of revolution around the vertical axis passing through the center of $\mathscr{C}$.

Assuming, without loss of generality, the center of $\mathscr{C}$ is the origin of $\mathbb{R}^3$ and consider the function $f(p)=\left<\varphi(p) \wedge \mathbf{e}_3,\eta(p)\right>$, $p \in \Sigma$, where $\wedge$ is the cross product of $\mathbb{R}^3$. A similar computation to the one in Theorem \ref{006.1} to obtain \eqref{0020} shows that $$\begin{dcases*}L_1f+2H_1H_2f=0 & in $\Sigma$ \\ \frac{\partial f}{\partial\nu}=0 & on $\partial\Sigma$\end{dcases*}.$$ The proof is finished if one can show that $f \equiv 0$.

Suppose, otherwise, that $f \not\equiv 0$. Then Lemma \ref{006.3} implies its nodal set $f^{-1}(\{0\})$ is a graph whose vertices are the critical points of $f$. We must show the nodal domain $\Sigma \backslash f^{-1}(\{0\})$ has at least $3$ connected components. If the function $f$ does not change its sign in a neighborhood of a point $p_0 \in f^{-1}(\{0\}) \cap \partial\Sigma$ then, as $L_1f=-2H_1H_2f$, the strong maximum principle \cite[Theorem 3.5]{gilbarg1977elliptic} and the Hopf Lemma \cite[Lemma 3.4]{gilbarg1977elliptic} implies that $\frac{\partial f}{\partial\nu}(p_0) \neq 0$ unless $f \equiv 0$ in a neighborhood of $p_0$, thus $f \equiv 0$ by Aronszajn's unique continuation principle \cite{aronszajn1956unique}. In both cases this leads to a contradiction, therefore the nodal domain has at least two connected components. The same arguments used in \cite[Theorem 3.1]{ainouz2016stable} to prove the nodal domain has a third connected component are valid here.

Denoting $\Sigma_1$ and $\Sigma_2$ two of these components, define the function $$\widetilde{f}=\begin{dcases*} f& in $\Sigma_1$ \\ \alpha f& in $\Sigma_2$ \\ 0 & in $\Sigma \backslash (\Sigma_1 \cup \Sigma_2)$\end{dcases*},$$ where $\alpha \in \mathbb{R}$ is such that $\widetilde{f} \in \mathcal{F}$. Since $\partial\Sigma_i \cap \partial\Sigma=\partial\Sigma \cap \Sigma_i$ and $\widetilde{f} \equiv 0$ outside $\Sigma_i$, we obtain
\begin{eqnarray*}
\int_{\Sigma_1} \left<P_1\nabla\widetilde{f},\nabla\widetilde{f}\right>-2H_1H_2\widetilde{f}^2\,d\mu_\Sigma &=& \int_{\Sigma_1} \left<P_1\nabla\widetilde{f},\nabla f\right>-2H_1H_2\widetilde{f}f\,d\mu_\Sigma \\
&=& -\int_{\Sigma_1} \widetilde{f}\left(L_1f+2H_1H_2f\right)\,d\mu_\Sigma+\int_{\partial\Sigma \cap \Sigma_1} \left\vert{P_1\nu}\right\vert\widetilde{f}\frac{\partial f}{\partial\nu}\,d\mu_\Sigma \\
&=& 0
\end{eqnarray*}
and, similarly, $\int_{\Sigma_2} \left<P_1\nabla\widetilde{f},\nabla\widetilde{f}\right>-2H_1H_2\widetilde{f}^2\,d\mu_\Sigma=0$. Thus, $$\mathcal{I}_1(\widetilde{f},\widetilde{f})=\sum_{i=1}^2 \int_{\Sigma_i} \left<P_1\nabla\widetilde{f},\nabla\widetilde{f}\right>-2H_1H_2\widetilde{f}^2\,d\mu_\Sigma=0.$$ and since $\Sigma$ is $r$-stable, Lemma \ref{005.3} implies that $\widetilde{f}$ is a Jacobi field on $\Sigma$. However, since $\widetilde{f}$ vanishes on $\Sigma \backslash (\Sigma_1 \cup \Sigma_2)$, it follows from Aronszajn's unique continuation principle that $\widetilde{f}=0$, which is a contradiction. Therefore $f \equiv 0$ and $\varphi(\Sigma)$ is a surface of revolution around the $x_3$-axis.
\end{proof}

\bibliographystyle{acm}
\bibliography{bibliography-articles}
\end{document}